\documentclass[11pt]{amsart}
\usepackage{amsmath,amsthm,amssymb}
\usepackage{fullpage}

\newtheorem{theorem}{Theorem}[section]
\newtheorem{proposition}{Proposition}[section]
\newtheorem{lemma}{Lemma}[section]

\newcommand{\dbar}{\overline{\partial}}
\newcommand{\ddbar}{\sqrt{-1}\partial\dbar}
\newcommand{\vp}{\varphi}
\newcommand{\ov}{\overline}
\renewcommand{\leq}{\leqslant}
\renewcommand{\geq}{\geqslant}
\renewcommand{\epsilon}{\varepsilon}

\title{Regularity of weak solutions of a complex Monge-Amp\`ere equation}

\author{G\'abor Sz\'ekelyhidi and Valentino Tosatti}

\begin{document}

\begin{abstract}
	We prove the smoothness of weak solutions to an elliptic complex
	Monge-Amp\`ere equation, using the smoothing property of the
	corresponding parabolic flow. 
\end{abstract}

\maketitle

\section{Introduction}
Let $(M,\omega)$ be a compact K\"ahler manifold. Our main result is the
following.
\begin{theorem}\label{thm:main}
	Suppose that $\vp\in PSH(M,\omega)\cap L^\infty(M)$ is a
	solution of the equation
	\[ (\omega + \ddbar\vp)^n = e^{-F(\vp,z)}\omega^n\]
	in the sense of pluripotential theory~\cite{BT76},
	where $F:\mathbf{R}\times M\to\mathbf{R}$ is smooth. Then $\vp$
	is smooth.
\end{theorem}

In particular if $M$ is Fano, $\omega\in c_1(M)$ and $h_\omega$
satisfies $\ddbar h_\omega=\mathrm{Ric}(\omega)-\omega$ then we can set
$F(\vp,z)=\vp-h_\omega$. The result then implies that
K\"ahler-Einstein currents with bounded potentials are in fact smooth. Such weak
K\"ahler-Einstein metrics were
studied by Berman-Boucksom-Guedj-Zeriahi in~\cite{BBGZ}, as part of their 
variational approach to complex Monge-Amp\`ere equations.

It follows from Ko\l odziej~\cite{Kol08} (see also \cite{GKZ}) that the solution $\vp$ in
Theorem \ref{thm:main} is automatically $C^\alpha$ for some
$\alpha>0$, but it does not seem possible to use this directly
to get further
regularity. The difficulty is that in the
equation 
\[ (\omega + \ddbar \vp)^n = e^f\omega^n,\]
the $C^1$ estimate for $\vp$ (due to B\l ocki~\cite{Bl09} and Hanani \cite{Han}) depends on a
$C^1$ bound for $f$, and in turn the Laplacian estimate for $\vp$ (due
to Yau~\cite{Yau78} and Aubin \cite{Aub}) depends on the Laplacian of $f$. 

To get around this difficulty we look at the corresponding parabolic
flow
\[ \frac{\partial \vp}{\partial t} =
\log\frac{(\omega+\ddbar\vp)^n}{\omega^n} + F(\vp,z).\]
Following the construction of Song-Tian \cite{ST09} for the K\"ahler-Ricci flow, we show that to find a solution for a short time, it is enough to have a
$C^0$ initial condition $\vp_0$ for which $(\omega+\ddbar\vp_0)^n$ is
bounded (see also \cite{CD, CT, CTZ} for earlier results, as well as \cite{Sim} for a weaker statement in the Riemannian case). The solution of the
flow will be smooth at any positive time. Then we need to argue
that if the initial condition $\vp_0$ is a weak solution of the
elliptic problem then the flow is stationary, so in fact $\vp_0$ is
smooth.

In Section \ref{sec:existence} we show that the flow (with smooth
initial data) exists for a short time, which only depends on a bound for
$\sup|\vp_0|$ and $\sup|\dot{\vp}_0|$. In Section~\ref{sec:main} we use
this to construct a solution to the flow with rough initial data, and
we prove Theorem~\ref{thm:main}.

\subsection*{Acknowledgements} We would like to thank D.H. Phong for support and encouragement as well as
J. Song and J. Sturm 
for very valuable help and B. Weinkove for useful comments. The second-named author is also grateful to S.-T. Yau for his support.
The question answered in Theorem \ref{thm:main} arose after R. Berman's talk
at the workshop ``Complex Monge-Amp\`ere Equation'' at Banff in October 2009, and we would like to thank him as well as the participants and organisers of that workshop for providing a stimulating working environment. The first-named author was supported in part by National Science Foundation grant DMS-0904223.

\section{Existence for the parabolic equation}\label{sec:existence}
In this section we consider the parabolic equation
\begin{equation}\label{eq:flow}
	\frac{\partial\vp}{\partial t} = \log\frac{ (\omega+\ddbar
\vp)^n}{\omega^n} + F(\vp,z),
\end{equation}
where $F:\mathbf{R}\times M\to\mathbf{R}$ is
smooth and we have the smooth
initial condition $\vp|_{t=0}=\vp_0$.  We write $\dot{\vp_0}$ for $\frac{\partial}{\partial t}\vp$ at $t=0$.

The main result of this section is the following
\begin{proposition}\label{prop:flow}
	There exist $T > 0$ depending only on $\sup |\vp_0|, \sup
	|\dot{\vp}_0|$ (and $\omega$ and $F$), such that there is a smooth
	solution $\vp(t,z):[0,T]\times M\to\mathbf{R}$ to Equation
	(\ref{eq:flow}). Moreover we also have smooth functions
	$C_k:(0,T]\to\mathbf{R}$ depending only on $\sup |\vp_0|,
	\sup|\dot{\vp}_0|$ such that
	\begin{equation}\label{eq:bounds}
		\Vert\vp(t)\Vert_{C^k(M)} < C_k(t) 
	\end{equation}
	as long as $t\leqslant T$. Note that $C_k(t)\to\infty$ as $t\to 0$.
\end{proposition}

The proof of the $C^1$ estimate is based on the arguments in B\l
ocki~\cite{Bl09} (see also \cite{Han, PS09}), whereas the $C^2$ estimate is
based on the Aubin-Yau second order estimate~\cite{Aub, Yau78} (see also \cite{ST09}
for the parabolic version we need here). The $C^3$ and higher order estimates follow the standard arguments in \cite{Yau78, Cao, PSS}, although there are a few new terms to control.

The existence of a smooth solution for $t\in[0,T')$ for some $T'>0$ that depends on the $C^{2,\alpha}$ norm of
$\vp_0$ is
standard. The aim is to obtain the estimates (\ref{eq:bounds}), which
allow us to extend the solution up to a time $T$, which only depends on
the initial condition in a weaker way.  We will write $\vp(t)$ for the
short time solution.

\begin{lemma}\label{lem:c0}
	There exists $T, C>0$ depending only on $\sup|\vp_0|$ and
	$\sup|\dot{\vp}_0|$ such that
	\begin{equation}\label{zero} |\vp(t)|, |\dot{\vp}(t)| < C,
\end{equation}
	as long as the solution exists and $t\leqslant T$. In particular 
	\begin{equation}\label{volest} \left |\log\frac{(\omega+\ddbar\vp)^n}{\omega^n}\right| <
	C\end{equation}
	for $t \leqslant T$.
\end{lemma}
\begin{proof}
For all $s$ let us define 
\[\begin{aligned}
	\overline{F}(s)&=\sup_{z\in M} F(s,z), \\
	\underline{F}(s)&=\inf_{z\in M} F(s,z),
\end{aligned} \]
which are continuous functions. If $M_t = \sup \vp(t,\cdot)$ and $m_t = \inf \vp(t,\cdot)$ then we obtain
\[\begin{aligned}
	\frac{dM_t}{dt}&\leqslant \overline{F}(M_t),\\
	\frac{dm_t}{dt}&\geqslant \underline{F}(m_t).
\end{aligned}\]
Comparing with the 
corresponding ODEs, we find that there exist $T, C>0$ depending only on
$m_0, M_0$ such that as long as our solution exists, and $t\leqslant T$, we have $\sup
|\vp(t)| < C$. 

Differentiating the equation we obtain
\begin{equation}\label{phidot}
\frac{\partial\dot{\vp}}{\partial t} = \Delta_\vp \dot{\vp} +
F'(\vp,z)\dot{\vp},
\end{equation}
where $F'$ is the derivative of $F$ with respect to the $\vp$ variable.
Since $F'(\vp,z)$ is bounded as long as $\vp$ is bounded, from the
maximum principle we get
\begin{equation}\label{eq:phidot}
	\sup|\dot{\vp}(t)| < \sup|\dot{\vp}(0)|e^{\kappa t},
\end{equation}
where $\kappa$ depends on $F$ and $\sup|\vp(0)|$. Hence for our choice
of $T$, we get
\[ \sup|\dot{\vp}(t)| < C,\]
for $t \leqslant T$, where $C$ depends on $\sup |\vp_0|$ and $\sup
|\dot{\vp}_0|$. 
\end{proof}

In the lemmas below $T$ will be the same as in the previous lemma.
\begin{lemma}\label{lem:c1} There exists $C>0$ depending on
	$\sup|\vp_0|$ and $\sup|\dot{\vp}_0|$ such that
	\begin{equation}\label{first} |\nabla\vp(t)|^2_\omega < e^{C/t},
\end{equation}
	as long as the solution exists and $t\leqslant T$ for the $T$ in Lemma
	\ref{lem:c0}.
\end{lemma}
\begin{proof}
	We modify B\l ocki's estimate~\cite{Bl09} for the complex
	Monge-Amp\`ere equation (cfr. \cite{Han}). Define
	\[ K = t\log |\nabla\vp|^2_\omega - \gamma(\vp),\]
	where $\gamma$ will be chosen later. Suppose that
	$\sup_{(0,t]\times M}K = K(t,z)$ is achieved. Pick normal coordinates for
	$\omega$ at $z$, such that $\vp_{i\bar j}$ is diagonal at this
	point (here and henceforth, indices will denote covariant derivatives with respect to the metric $\omega$). Let us write $\beta = |\nabla\vp|^2_\omega$ and $\Delta_\vp$ for the Laplacian of the metric $\omega+\ddbar\vp$. There
	exists $B>0$ such that
	\[\begin{gathered}
		0 \leqslant \left(\frac{\partial}{\partial t} -
		\Delta_\vp\right)K \leqslant -\frac{t}{\beta}\sum_{i,p}\frac{
		|\vp_{ip}|^2 + |\vp_{i\bar p}|^2}{1 + \vp_{p\bar p}}
		+ (t^{-1}(\gamma')^2 + \gamma'')\sum_{p}\frac{
		|\vp_p|^2}{1+\vp_{p\bar p}} \\
		- (\gamma' -Bt)\sum_{p}\frac{1}{1+\vp_{p\bar p}} +
		\log\beta + \frac{Ct}{\beta} -\gamma'\dot{\vp} +
		n\gamma' + Ct.
	\end{gathered}\]
	The constant $C$ depends on bounds for $F$ and $F'$, and also we
	used that $\nabla K=0$ at $(t,z)$. 

	Now we apply B\l ocki's trick to get rid of the term containing
	$(\gamma')^2$. At $(t,z)$ we have
	\[ t\beta_p = \gamma'\beta\vp_p,\]
	where
	\[ \beta_p = \vp_p\vp_{p\bar p} + \sum_j \vp_{jp}\vp_{\bar
	j},\]
	remembering that $\vp_{j\bar p}$ is diagonal. It follows that 
	\[ \sum_j \vp_{jp}\vp_{\bar j} = (t^{-1}\gamma'\beta -
	\vp_{p\bar p})\vp_p,\]
	and so
	\[ \begin{aligned}
		\frac{t}{\beta}\sum_{j,p}\frac{|\vp_{jp}|^2}{1+\vp_{p\bar
		p}} &\geqslant \frac{t}{\beta^2}\sum_p\frac{\left|\sum_j
		\vp_{jp}\vp_{\bar j}\right|^2}{1+\vp_{p\bar p}}\\
		&= \frac{t}{\beta^2}\sum_p \frac{|t^{-1}\gamma'\beta -
		\vp_{p\bar p}|^2|\vp_p|^2}{1+\vp_{p\bar p}} \\
		&\geqslant
		t^{-1}(\gamma')^2\sum_p\frac{|\vp_p|^2}{1+\vp_{p\bar p}} -
		2\gamma',
	\end{aligned}\]
	where we assume that $\gamma'>0$. 
	Also from Lemma \ref{lem:c0} we know that $\dot{\vp}$ is bounded.
	Combining these estimates we obtain 
	\[ 0\leqslant \gamma''\sum_p\frac{ |\vp_p|^2}{1+\vp_{p\bar p}}
	- (\gamma'-Bt)\sum_p\frac{1}{1+\vp_{p\bar p}} + \log\beta +
	\frac{Ct}{\beta} + C\gamma' + Ct.\]
	We now choose $\gamma(s) = As - \frac{1}{A}s^2$. 
	We can assume that $\log\beta > 1$ at $(t,z)$, so in particular
	$\frac{t}{\beta}$ is bounded above as long as $t<T$. Then if $A$
	is chosen sufficiently large, we get a constant $C'>0$ such that
	\begin{equation}\label{eq:1}
		\sum_p\frac{1}{1+\vp_{p\bar p}} + \sum_p\frac{|\vp_p|^2}
		{1+\vp_{p\bar p}} \leqslant C'\log\beta,
	\end{equation}
	so in particular $(1+\vp_{p\bar p})^{-1}\leqslant C'\log\beta$ for
	each $p$. 
	From \eqref{volest} we know that 
	\[ \prod_p (1+\vp_{p\bar p}) < C,\] 
	so
	\[ 1+\vp_{p\bar p} \leqslant C(C'\log\beta)^{n-1}, \]
	and using (\ref{eq:1}) we get
	\[ \beta = \sum_p |\vp_p|^2 \leqslant C(C'\log \beta)^n.\]
	This shows that $\beta< C$ and in turn $K< C$
	for some constant $C$. So either $K$ achieves a maximum for some
	$t>0$ in which case we have just bounded it, or it achieves its
	maximum for $t=0$, which is bounded in terms of $\sup
	|\vp_0|$.
\end{proof}

From now on, let us write $g$ for the metric $\omega$ and $g_\vp$ for the
	metric $\omega+\ddbar\vp$.

\begin{lemma}\label{lem:c2} There exists $C>0$ depending on
	$\sup|\vp_0|$ and $\sup|\dot{\vp}_0|$ such that
\begin{equation}\label{second}	
 0<\mathrm{tr}_g(g_\vp)=n+\Delta_g \vp(t) < e^{Ce^{C/t}},
\end{equation}
	as long as the solution exists and $t \leqslant T$, for the $T$ from
	Lemma \ref{lem:c0}.
\end{lemma}
\begin{proof}
	We let
	\[ H = e^{-\frac{\alpha}{t}}\log \mathrm{tr}_g(g_\vp) -
	A\vp,\]
	where $\alpha=C$ from Lemma \ref{lem:c1} and $A$ is chosen
	later. In particular we will use that
	$e^{-\alpha/t}|\nabla\vp|^2_g < 1$.
	Standard calculations (from Aubin and Yau~\cite{Aub,Yau78})
	show that there exist $B>0$ such that
	\[ \Delta_\vp \log\mathrm{tr}_g(g_\vp) \geqslant
	-B\mathrm{tr}_{g_\vp}g -
	\frac{\mathrm{tr}_g\mathrm{Ric}(g_\vp)}{\mathrm{tr}_g(g_\vp)}.\]
	Using this we can compute
	\begin{equation}\label{eq:2}
		\begin{aligned}
		\left(\frac{\partial}{\partial t} - \Delta_\vp\right)H
		&\leqslant
		\frac{\alpha e^{-\alpha/t}}{t^2}\log\mathrm{tr}_g
		(g_\vp) + \frac{Ce^{-\alpha/t}}{\mathrm{tr}_g(g_\vp)}+
		\frac{e^{-\alpha/t}\Delta_g
		F(\vp,z)}{\mathrm{tr}_g(g_\vp)} \\
		&\quad +Be^{-\alpha/t}\mathrm{tr}_{g_\vp}g -
		A\dot{\vp} + An - A\mathrm{tr}_{g_\vp}g.
		\end{aligned}
	\end{equation}
	Here
	\[ \Delta_g F(\vp,z) = \Delta_g F + 2\mathrm{Re}(g^{i\bar
	j}F'_i\vp_{\bar j}) + F'\Delta_g\vp+F''|\nabla\vp|^2_g,\]
	where $F'$ is the derivative in the $\vp$ variable, and
	$\Delta_g F$ is the Laplacian of $F(\vp,z)$ in the $z$
	variable. So we have constants $C_1,C_2,C_3$ such that
	\[ \Delta_g F(\vp,z)\leqslant C_1 + C_2|\nabla\vp|^2_g +
	C_3\mathrm{tr}_g(g_\vp).\]
	From \eqref{volest} we have bounds on above and below on
	$\frac{\det g_{\vp}}{\det g}$, so for some constant $C$ we have
	$\mathrm{tr}_g(g_\vp) > C^{-1}$ and also
	$\mathrm{tr}_g(g_\vp)\leqslant C(\mathrm{tr}_{g_\vp}g)^{n-1}$.
	Using these in (\ref{eq:2}) we get
	\[
	\begin{aligned}
		\left(\frac{\partial}{\partial t} - \Delta_\vp\right)H
	&\leqslant -(A-Be^{-\alpha/t})\mathrm{tr}_{g_\vp}g +
	C\log\mathrm{tr}_{g_\vp}g + C \\
	&\leqslant -(A - C- Be^{-\alpha/t})\mathrm{tr}_{g_\vp}g + C',
	\end{aligned}\]
	as long as $t \leqslant T$. Choosing $A$ large enough, we can use the
	maximum principle to bound $H$ in terms of its value for $t=0$,
	which is bounded by $\sup|\vp_0|$. 
\end{proof}

We note here that if one is interested in the special case of weak K\"ahler-Einstein currents (i.e. $F=\vp-h_\omega$), then the gradient estimate in Lemma \ref{lem:c1} is not needed.
We now describe how to get the higher order estimates, as long as the solution exists and $t \leqslant T$, 
for the $T$ from
	Lemma \ref{lem:c0}.
As in \cite{Yau78} we let $\vp_{i\ov{j}k}$ be the third covariant derivative of $\vp$ with respect to the Levi-Civita connection of $\omega$, and we define
$$S=g_\vp^{i\ov{p}}g_\vp^{q\ov{j}}g_\vp^{k\ov{r}} \vp_{i\ov{j}k}\vp_{\ov{p}q\ov{r}}.$$
From now on, we will denote by $C(t)$ a smooth real function defined on $(0,T]$, which is allowed to blow up when $t$ approaches zero, which depends only on $\sup |\vp_0|,
	\sup|\dot{\vp}_0|$ and which may vary from line to line. These functions $C(t)$ can be made completely explicit.
Using \eqref{second} it is clear that an estimate of the form $S\leq C(t)$
implies an estimate of the form $\|\vp(t)\|_{C^{2+\alpha}(g)}\leq C(t)$, for any $0<\alpha<1$.
To estimate $S$ we first compute its evolution. It is convenient to use the general computation by Phong-\v Se\v sum-Sturm \cite{PSS}, which uses the following notation.
We denote by $h^i_j=g^{i\ov{k}}(g_{j\ov{k}}+\vp_{j\ov{k}})$, which is an endomorphism of the tangent bundle. Then 
$S$ can be written in terms of the connection $\nabla h h^{-1}$ as
$$S=g_\vp^{p\ov{q}}g_{\vp, i\ov{j}}g_\vp^{k\ov{\ell}} (\nabla_p h h^{-1})^i_k \ov{(\nabla_q h h^{-1})^j_\ell}=|\nabla h h^{-1}|^2_{g_\vp},$$
where $\nabla$ is the Levi-Civita connection of $\omega_\vp$. Then the computations in \cite{PSS} yield
\[
\begin{aligned}
		\left(\frac{\partial}{\partial t} - \Delta_\vp\right)S
		&=-|\nabla(\nabla h h^{-1})|^2_{g_\vp}-|\ov{\nabla}(\nabla h h^{-1})|^2_{g_\vp}\\
&\quad + 2\mathrm{Re} \left\langle (\nabla T - \nabla R ,\nabla h h^{-1}\right\rangle_{g_\vp}\\
&\quad +(\nabla_p h h^{-1})^i_k \ov{(\nabla_q h h^{-1})^j_\ell}(T^{p\ov{q}}g_{\vp, i\ov{j}}g_\vp^{k\ov{\ell}}-g_\vp^{p\ov{q}}T_{i\ov{j}}g_\vp^{k\ov{\ell}}+g_\vp^{p\ov{q}}g_{\vp, i\ov{j}}T^{k\ov{\ell}}),
		\end{aligned}\]
where $T_{i\ov{j}}=-\left(\frac{\partial}{\partial t}g_{\vp}+\mathrm{Ric}(g_\vp)\right)_{i\ov{j}},$ $(\nabla T)^p_{qr}=g_\vp^{p\ov{s}}\nabla_q T_{r\ov{s}}$, $(\nabla R)^p_{qr}=g_\vp^{s\ov{t}}\nabla_s R^p_{rq\ov{t}}$ and $R^p_{rq\ov{t}}$ is the curvature of the fixed metric $g$. Along the standard K\"ahler-Ricci flow the tensor $T$ vanishes, while in our case differentiating \eqref{eq:flow} we get
\begin{equation}\label{tensor}
-T_{i\ov{j}}=\mathrm{Ric}(g)_{i\ov{j}}+F''\vp_i\vp_{\ov{j}} +F'\vp_{i\ov{j}}+F_{i\ov{j}}+2\mathrm{Re}(F'_i\vp_{\ov{j}}).
\end{equation}
Using \eqref{first} and \eqref{second} we can then estimate
$$|(\nabla_p h h^{-1})^i_k \ov{(\nabla_q h h^{-1})^j_\ell}(T^{p\ov{q}}g_{\vp, i\ov{j}}g_\vp^{k\ov{\ell}}-g_\vp^{p\ov{q}}T_{i\ov{j}}g_\vp^{k\ov{\ell}}+g_\vp^{p\ov{q}}g_{\vp, i\ov{j}}T^{k\ov{\ell}})|\leq C(t) S.$$
The term $2\mathrm{Re}\langle\nabla R ,\nabla h h^{-1}\rangle_{g_\vp}$ is comparable to $S$, but bounding $2\mathrm{Re}\langle\nabla T ,\nabla h h^{-1}\rangle_{g_\vp}$ requires a bit more work.
Differentiating \eqref{tensor} and using \eqref{zero}, \eqref{first} and \eqref{second} we see that all the terms in $2\mathrm{Re}\langle\nabla T ,\nabla h h^{-1}\rangle_{g_\vp}$ are comparable to $C(t) S$ except for two terms of the form
$$\langle \vp_{ij} g_\vp^{k\ov{\ell}}\vp_{\ov{\ell}}, (\nabla_i h h^{-1})^k_j\rangle_{g_{\vp}}.$$
We bound these by $|\vp_{ij}|^2_{g_{\vp}} + C(t)S,$ so overall we get
$$\left(\frac{\partial}{\partial t} - \Delta_\vp\right)S\leq C(t)S +|\vp_{ij}|^2_{g_{\vp}}+C.$$
The term $C(t)S$ can be controlled by using $\mathrm{tr}_g(g_\vp)$ in the usual way (cfr. \cite{PSS}). For the term $|\vp_{ij}|^2_{g_{\vp}}$ we note that using \eqref{zero}, \eqref{first} and \eqref{second} we have
\begin{equation*}
\begin{aligned}\left(\frac{\partial}{\partial t} - \Delta_\vp\right)|\nabla\vp|^2_g&\leq
-\sum_{i,p}\frac{
		|\vp_{ip}|^2 + |\vp_{i\bar p}|^2}{1 + \vp_{p\bar p}}+2\mathrm{Re}\langle \nabla\vp,
		F'\nabla\vp+\nabla F\rangle_g +C\mathrm{tr}_{g_\vp}g|\nabla\vp|^2_g\\
		&\leq -\frac{|\vp_{ij}|^2_{g_{\vp}}}{C(t)}+C(t).
	\end{aligned}
	\end{equation*}
We can then apply the maximum principle to the quantity
$$G=\frac{S}{C_1(t)}+\frac{\mathrm{tr}_g(g_\vp)}{C_2(t)}+\frac{|\nabla\vp|^2_g}{C_3(t)},$$
for suitable functions $C_i(t)$ that depend only on the given data, and get
$G\leq C$, which implies the desired estimate for $S$. This means that as long as the solution exists and $0<t\leqslant T$ we have a bound on $\Vert\vp(t)\Vert_{C^{2+\alpha}(M)}$.
Since by standard parabolic theory one can start the flow with initial data in $C^{2+\alpha}$,
this shows that the flow has a $C^{2+\alpha}$ solution defined on $[0,T]$.

The next step is to estimate $\sup|\ddot{\vp}(t)|$ and $\sup|\partial_i\partial_{\ov{j}}\dot{\vp}(t)|$. 
It is easy to see that both of these quantities are bounded if we bound $|\mathrm{Ric}(g_\vp)|_{g_\vp}$. Following the computation in \cite[(6.31)]{PSSW} one can derive the following estimate (there are essentially no new bad terms in this case)
$$\left(\frac{\partial}{\partial t} - \Delta_\vp\right)|\mathrm{Ric}(g_\vp)|_{g_\vp}\leq C(t)|\mathrm{Rm}(g_\vp)|^2+C(t).$$
From one of the two good positive terms in the evolution of $S$ we get
$$\left(\frac{\partial}{\partial t} - \Delta_\vp\right)S\leq -\frac{|\mathrm{Rm}(g_\vp)|^2}{C(t)}+C(t)$$
and so the maximum principle applied to the quantity $\frac{|\mathrm{Ric}(g_\vp)|_{g_\vp}}{C_1(t)}+\frac{S}{C_2(t)}$ gives the desired bound $|\mathrm{Ric}(g_\vp)|_{g_\vp}\leq C(t)$.

It now follows from the parabolic Schauder estimates applied to (\ref{phidot}) that we have bounds for $\vp$ in the parabolic H\"older space $C^{2+\alpha, 1+\alpha/2}(M\times[\varepsilon,T])$ for any $\varepsilon>0$, with the bounds only depending on $\varepsilon$, 
$\sup|\vp_0|$ and $\sup|\dot{\vp}_0|$.
By the parabolic Schauder estimates we then also get bounds on all higher order derivatives for $\vp$, and letting $\varepsilon\to 0$ we get the required bounds on $\vp(t)$ that blow up as $t$ goes to zero. In particular, we get a smooth solution $\vp(t)$ that exists on $[0,T]$, with bounds as in \eqref{eq:bounds}. This completes the proof of Proposition \ref{prop:flow}.

\section{Proof of Theorem \ref{thm:main}}\label{sec:main}
Suppose that $\vp$ is a bounded $\omega$-plurisubharmonic solution of the
equation
\begin{equation}\label{eq:MA}
	(\omega + \ddbar\vp)^n = e^{-F(\vp,z)}\omega^n,
\end{equation}
where $F$ is a smooth function.
First of all we want to prove existence of the flow \eqref{eq:flow} with rough initial data 
$\vp$. For this, we follow the proof of Song-Tian \cite{ST09} in the case of K\"ahler-Ricci flow.

It follows from Ko\l odziej~\cite{Kol98}
that in this case $\vp$ is continuous (in fact it is even $C^\alpha$ \cite{GKZ, Kol08}). 
Let us approximate $\vp$ with a sequence of smooth
functions $u_k$, such that
\begin{equation}\label{eq:approx}
	\sup_M |\vp - u_k| \to 0,
\end{equation}
as $k\to\infty$. By Yau's theorem~\cite{Yau78} there are smooth
functions $\psi_k$ such that
\begin{equation}\label{monge} (\omega + \ddbar\psi_k)^n = c_k e^{-F(u_k,z)}\omega^n,
\end{equation}
where the positive constants $c_k$ are chosen so that the integrals of both sides of \eqref{monge} match. When $k$ is large we see that $c_k$ approaches $1$.
Moreover, we can normalise the solution $\psi_k$ so that
\[ \sup_M (\psi_k - \vp) = \sup_M (\vp - \psi_k).\]
Using \eqref{eq:approx} together with Ko\l odziej's stability result~\cite{Kol03} we obtain 
\begin{equation} \label{eq:stability}
	\lim_{k\to\infty} \Vert \psi_k -\vp\Vert_{L^\infty} = 0.
\end{equation}
Using Proposition \ref{prop:flow} we can solve the equation
\begin{equation}\label{eq:flow1}
	\frac{\partial \vp_k}{\partial t} = \log\frac{ (\omega+\ddbar
\vp_k)^n}{\omega^n} + F(\vp_k,z)-\log c_k,
\end{equation}
with initial condition $\vp_k|_{t=0}=\psi_k$ for a short time $t\in [0,T]$
independent
of $k$, since by \eqref{eq:approx}, \eqref{monge} and (\ref{eq:stability}) we have
uniform bounds on the initial data
$\sup|\psi_k|$ and $\sup|\dot{\vp}_k(0)|$. As in ~\cite{ST09} we have

\begin{lemma}
	The sequence $\vp_k$ is a Cauchy sequence in $C^0([0,T]\times
	M)$, ie. 
	\[ \lim_{j,k\to\infty} \Vert \vp_j -
	\vp_k\Vert_{L^\infty([0,T]\times M)} = 0.\]
\end{lemma}
\begin{proof}
Fix $j,k$ and let $\mu = \vp_j - \vp_k$. Then
\[ \frac{\partial \mu}{\partial t} = \log\frac{(\omega +
\ddbar\vp_k +
\ddbar \mu)^n}{(\omega + \ddbar\vp_k)^n} +
F(\vp_j,z)-F(\vp_k,z)+\log\frac{c_k}{c_j},\]
and $\mu|_{t=0} = \psi_j-\psi_k$. 
At any time given time $t$, the maximum of $\mu$ is achieved at some point $z\in M$, and at $z$ we have
\[ \frac{ d \mu_{\mathrm{max}}}{dt} \leqslant F(\vp_j(t,z),z)-F(\vp_k(t,z),z)+\log\frac{c_k}{c_j}\leq
\kappa |\mu(z)|+\log\frac{c_k}{c_j},\]
where $\kappa$ is independent of $j,k$. Similarly, at the point $z'$ where the minimum of $\mu$ is achieved, we have
\[ \frac{ d \mu_{\mathrm{min}}}{dt} \geq -\kappa |\mu(z')|+\log\frac{c_k}{c_j}.\]
Putting these together we see that
\[\frac{ d |\mu|_{\mathrm{max}}}{dt} \leq \kappa |\mu|_{\mathrm{max}}+\left|\log\frac{c_k}{c_j}\right|,\]
where the derivative is interpreted as the limsup of the difference quotients at the points where it does not exist.
 It follows that
\[  \sup_{[0,T]\times M}| \vp_j -
	\vp_k|\leqslant e^{\kappa T} \left(\Vert \psi_j -
\psi_k\Vert_{L^\infty(M)}+\frac{1}{\kappa} \left|\log \frac{c_k}{c_j}\right|\right)-\frac{1}{\kappa}\left|\log\frac{c_k}{c_j}\right|.\]
Now (\ref{eq:stability}) and the fact that $c_k$ converges to $1$ imply the result.
\end{proof}

Using this lemma we can define
\[ \Phi = \lim_{j\to\infty} \vp_j,\]
which is in $C^0([0,T]\times M)$. 
Moreover from Proposition \ref{prop:flow} for any $\epsilon>0$
we have uniform bounds on all derivatives of the $\vp_j$ for
$t\in[\epsilon,T]$, so in fact for all $k$ we have
\[ \lim_{j\to\infty} \Vert \Phi - \vp_j\Vert_{C^k(M\times[\epsilon,T])} = 0.\]
From Equation (\ref{eq:phidot}) we get
\[ \sup_M |\dot{\vp}_k(t)| < C\sup_M |\dot{\vp}_k(0)|\]
for $t\in [0,T)$, but from (\ref{eq:flow1}) we have
\[ \dot{\vp}_k(0) = \log\frac{(\omega+\ddbar\psi_k)^n}{\omega^n} +
F(\psi_k,z)-\log c_k = F(\psi_k,z) - F(\vp_k,z)-\log c_k,\]
which converges to zero when $k$ goes to infinity.
It follows that for any $t > 0$ we have
\[ \dot{\Phi}(t) = \lim_{j\to\infty} \dot{\vp}_j(t) = 0. \]
Hence $\Phi$ is constant on $(0,T]$, but since it is continuous on $[0,T]$
 it follows that $\Phi(t)=\Phi(0)$ for all $t \leqslant T$. But
$\Phi(0)$ is our solution $\vp$ of Equation (\ref{eq:MA}), whereas
$\Phi(t)$ is smooth for $t>0$. Hence $\vp$ is smooth.

\bigskip
\noindent
Mathematics Department\\ Columbia University\\ New York, NY 10027

\end{document}